\documentclass[12pt,a4paper]{article}
\usepackage[latin1]{inputenc}
\usepackage[english]{babel}
\usepackage{amsmath}
\usepackage{amsfonts}
\usepackage{amsthm}
\usepackage{amssymb}
\usepackage{mathrsfs}
\usepackage{graphicx}
\usepackage{color}
\usepackage[margin=2.5cm]{geometry}

\newtheorem{theorem}{Theorem}[section]
\newtheorem{lemma}{Lemma}[section]
\newtheorem{proposition}{Proposition}[section]

\newtheorem{remark}{Remark}[section]

\newtheorem{definition}{Definition}[section]

\newcommand{\R}{{\mathbb R}}

\def\phi{\varphi}

\def\phi{\varphi}

\def\phi{\varphi}

\def\beq{\begin{equation}}
\def\eeq{\end{equation}}
\def\beqn{\begin{eqnarray}}
\def\eeqn{\end{eqnarray}}

\numberwithin{equation}{section}
\numberwithin{figure}{section}

\title{Asymptotics for time-changed diffusions}
\date{\today}

\author{\textsc{Raffaela Capitanelli}\footnote{
Dipartimento di Scienze di Base e Applicate per l'Ingegneria, Sapienza"  Universit\`{a}  di Roma,
Via A. Scarpa 16,  00161 Roma, Italy} \ \
\& \ \
\textsc{Mirko D'Ovidio}$^\ast$\\
{\small raffaela.capitanelli@uniroma1.it, mirko.dovidio@uniroma1.it}
}

\begin{document}

\maketitle

\begin{abstract}
We consider time-changed diffusions driven by generators with discontinuous coefficients. The PDE's connections are investigated and in particular some results on the asymptotic analysis according to the behaviour of the coefficients are presented. 
\end{abstract}

%\tableofcontents

\section{Introduction}
\label{Sec-one}

We consider the balls $\Omega_l \subset \Omega_\ell \subset \Omega_r$ centered at the same point where $\Omega_q$ is a ball with radius $q=l,\ell,r$ where $r=\ell+\epsilon$ and $\epsilon >0$.  Let us introduce two independent Brownian motions $B^1, B^2$ and the process (see the generator \eqref{A-form} below)
\begin{align*}
B^{(\alpha, \lambda)}_t := \left\lbrace
\begin{array}{ll}
\displaystyle B^1_t & \textrm{on } \Omega_\ell \setminus \overline{\Omega_l} \\
\displaystyle B^2_{\lambda t} & \textrm{on } \Sigma_\epsilon = \Omega_r \setminus \overline{\Omega_\ell}
\end{array}
\right. 
\end{align*}
with skew condition on $\partial \Omega_\ell$:
\begin{align*}
\forall\, x \in \partial \Omega_\ell, \qquad \mathbb{P}_x(B^{(\alpha, \lambda)}_t \in \Omega_\ell \setminus \overline{\Omega_l}) =1-\alpha \quad \textrm{and} \quad \mathbb{P}_x(B^{(\alpha, \lambda)}_t \in \Sigma_\epsilon) = \alpha.
\end{align*}
Moreover, $B^{(\alpha, \lambda)}$ is killed on $\partial \Omega_l$ and $\partial \Omega_r$. Since we have different variances depending on $\lambda>0$, we refer to the process $B^{(\alpha, \lambda)}$ as a modified process. Obviously, $\alpha \in (0,1)$ is the skewness parameter and $B^{(\alpha, \lambda)}$ is called skew process. We write $\alpha=\alpha_\epsilon$, $\lambda=\lambda_\epsilon$ by underling the dependence from $\epsilon$ ($\alpha_\epsilon\to 0$ and $\lambda_\epsilon\to 0$ as $\epsilon\to 0$) and we consider the collapsing domain $\Sigma_\epsilon \subset \mathbb{R}^d$ with $d=1,2$ (that is, with vanishing thickness $\epsilon$). 

Our aim is to study  a killed diffusion on $\Omega_r \setminus \overline{\Omega_l}$ with skew condition on $\partial \Omega_\ell$ and different behaviour in $\Sigma_\epsilon=\Omega_r \setminus \overline{\Omega_\ell}$ and $\Omega_\ell \setminus \overline{\Omega_l}$  under the assumption \begin{align}
\lim_{\epsilon\to 0} \frac{\alpha \epsilon}{\lambda} =0.
\label{main-condition}
\end{align}
For the classical case $\alpha=\lambda$ which has been extensively investigated in literature (see for example \cite{AB,BCF} and the references therein), the condition \eqref{main-condition} becomes trivial. 

Our new result is given by the asymptotic analysis obtained under \eqref{main-condition} with $\alpha\neq\lambda$. \\

We first consider time-changed Brownian motions on the line. After that we pass to Brownian motion on balls and the associated  skew-product representations $(R, \Theta)$. In particular, we study the time-changed Bessel process and the winding number $\Theta$ can be neglected by exploiting the isotropy of the Brownian diffusions. Thus,  we focus on the radial parts $R$ which are driven by a Markov generators with representation given in terms of scale functions and speed measures. The eigenfunction expansions of such generators of these Markov processes and their associated semigroups can be explicitly computed as solutions to the corresponding Sturm-Liouville problem  (\cite{dovSL,LMS,LS}). 

We obtain the following result.

\begin{theorem}
\label{main-thm}
Let $X_t$ be a reflecting Brownian motion on $\overline{\Omega_\ell} \setminus \overline{\Omega_l}$ with boundary local time
\begin{align*}
L^{\partial \Omega_\ell}_t (X) := \int_0^t \mathbf{1}_{\partial \Omega_\ell} (X_s)ds.
\end{align*}
Under \eqref{main-condition}, we have that
%\begin{align*}
%\lim_{\epsilon \to 0}\mathbb{E}_x \left[ \int_0^{\tau_{D_\epsilon}} \mathbf{1}_{\mathbb{R}^2}(B^{(\alpha, \lambda)}_t)dt \right] = \mathbb{E}_x\left[ \int_0^\infty  \mathbf{1}_{\mathbb{R}^2}(X_t) \exp \left( - \left( \lim_{\epsilon \to 0} \frac{\alpha}{(1-\alpha) \epsilon} \right) L^{\partial D}_t(X) \right) \right]
%\end{align*}
\begin{align}
\label{eq-main-thm}
\lim_{\epsilon \to 0}\mathbb{E}_x \left[ \int_0^\infty f(B^{(\alpha, \lambda)}_t) M^\epsilon_t dt \right] = \mathbb{E}_x\left[ \int_0^{\tau_l}  f(X_t) \exp \left( - \left( \lim_{\epsilon \to 0} \frac{\alpha}{(1-\alpha) \epsilon} \right) L^{\partial \Omega_\ell}_t(X) \right) dt \right]
\end{align}
where $M^\epsilon_t := \mathbf{1}_{(t < \tau_{\epsilon})}$ with $\tau_{\epsilon} := \inf \{s > 0\,:\, B^{(\alpha, \lambda)}_s \notin \Omega_r \setminus \overline{\Omega_l}\}$ and $\tau_{l} = \inf \{s > 0\,:\, X_s \in \partial \Omega_l\}$.
\end{theorem}

An interesting connection is related to a conjecture of Feller. An elastic Brownian motion on $[0, \infty)$ with condition $\alpha u(0) = (1-\alpha) u^\prime(0)$, $\alpha \in (0,1)$ is identical in law to a reflecting Brownian motion  $B^+$ killed according with the conditional law $\mathbb{P}(T > t | B^+) = \exp( - \frac{\alpha}{1-\alpha} L^{\partial \Omega_\ell}_t)$.  We notice that the special cases $\alpha=1$ or $\alpha=0$ correspond to Dirichlet or Neumann conditions. For example, for $\alpha=1$, the process $B^{(1,\lambda)}$ moves on $\Omega_r \setminus \Omega_\ell$ collapsing to $\partial \Omega_\ell$ as the thickness $\epsilon\to 0$. Thus, we consider only starting point $x \in \partial \Omega_\ell$. On the other hand, $X_t$ moves on $\overline{\Omega_\ell}$. If $X_t$ is forced to start at $x \in \partial \Omega_\ell$, the local time is positive and the right-hand side of \eqref{eq-main-thm} equals $0$, that is, we have the Dirichlet condition on $\partial \Omega_\ell$. The Neumann boundary condition follows immediately by considering $B^{(0, \lambda)} \in \overline{\Omega_\ell} \setminus \overline{\Omega_l}$ and the right-hand side of \eqref{eq-main-thm} for $\alpha=0$.

The plan of the work is the following. In Section 2 we give some preliminaries about the characterization of one-dimensional diffusions by means of scale functions and speed measures. In Section  3 we study the one-dimensional Brownian motion by exploiting the technique based on the corresponding scale function and speed measure. We provide some explicit results about the mean exit time. In Section 4 we study the 2-dimensional Brownian motion by skew-product representation and therefore we apply the technique introduced in Section 2 by considering the associated Bessel process. We provide some results on the asymptotic behaviour of the mean exit time.  In Section 5 we give the proof of Theorem \ref{main-thm}. Finally, in the last section we consider some possible extensions of our results to irregular domains and delayed diffusions.

\section{Preliminaries}
A one-dimensional stochastic process $Y_t$, $t\geq 0$ can be characterized by means of the scale function $S(\cdot)$ and the speed measure $M(\cdot)$ (see for example \cite{KarTay81}). For a one-dimensional diffusion with generator
\begin{equation}
A=\frac{\rho(x)}{2}\frac{d}{d x}\left( a(x) \frac{d}{dx} \right) + b(x)\frac{d}{dx} \label{gen-op-diff}
\end{equation}
we introduce
\begin{equation}
\label{scalePrel}
\mathfrak{s}(x)= S^\prime(x)= \frac{1}{a(x)} \exp\left( -2\int^x \frac{b(\eta)}{\rho(\eta)a(\eta)}d\eta\right)
\end{equation}
and the speed density
\begin{equation}
\mathfrak{m}(x)= \left( \rho(x) a(x) \mathfrak{s}(x)\right)^{-1}
\end{equation}
such that
\begin{equation*}
\mathbb{P}_x (Y_t \in \cdot) = \int_{\cdot} K(t,x,y) \mathfrak{m}(dy)
\end{equation*}
and $\partial_t K(t,x,\cdot) = AK(t,x,\cdot)$. The corresponding integrals are therefore written as
\begin{equation*}
S(x) = \int^x \mathfrak{s}(\eta)d\eta
\end{equation*}
for the scale function $S$ and 
\begin{equation*}
M(x) = \int^x \mathfrak{m}(\eta)d\eta 
\end{equation*}
for the speed measure $M$. We also write
\begin{equation*}
M(l,x]=M(x) - M(l) \quad \textrm{and} \quad S(l,x] = S(x) - S(l)
\end{equation*}
in order to maintain the notation in \cite{KarTay81}. 
Moreover, we notice that, from \eqref{scalePrel},
\begin{equation*}
AS=\frac{\rho(x)}{2}\frac{d}{d x}\left( a(x) \frac{dS}{dx} \right) + b(x)\frac{dS}{dx}= \frac{\rho(x)}{2}\frac{d}{d x}\bigg( a(x) \mathfrak{s}(x) \bigg) + b(x)\mathfrak{s}(x)=0. 
\end{equation*}

Let $B_t$, $t\geq 0$ be a Brownian motion on the real line and denote by
\begin{equation*}
P_tf(x)=\mathbb{E}_x f(B_t)=\mathbb{E}f(x+B_t) = \int_\mathbb{R} f(y)K(t,x,y)dy
\end{equation*}
its semigroup (with $K(t,x,y)=K(t,y-x)$, the heat kernel). Let $B^\alpha_t$, $t\geq 0$ be a skew Brownian motion on the real line with $\alpha\in (0,1)$. A deep discussion about skew Brownian motion can be found in \cite{Lejay06} and the references therein. The reader can also consult the interesting papers  \cite{HarrShep81, ORT13, Ram11}. The skew Brownian motion is the solution to $dB^\alpha_t = \sigma\, dB_t + (2\alpha -1) dL_t^{\{\ell\}}$ where $B$ is a standard Brownian motion and the symmetric local time (at $\ell \in \mathbb{R}$) of $B^\alpha$ is given by
\begin{equation}
L^{\{\ell\}}_t = \frac{1}{2} \left( L^{\ell^-}_t + L^{\ell^+}_t \right)
\end{equation}
where
\begin{equation}
 L^{\ell^-}_t = 2(1-\alpha) L^{\{\ell\}}_t,  \quad  L^{\ell^+}_t =2 \alpha L^{\{\ell\}}_t.
\end{equation}
We refer to $L^{\ell^-}_t$ and $L^{\ell^+}_t$ respectively as left and right local time. Let $u(x,t)=\mathbb{E}_x[f(B^\alpha_t)]$ (see for example \cite{walsh78,ABTWW11} for the explicit representation of the kernel $K^\alpha$ for $B^\alpha$). Then, $u$ is the solution to
\begin{equation}
\partial_t u(x,t) = \frac{\sigma^2}{2}\triangle u(x,t), \quad x \in \mathbb{R}, \; t>0
\end{equation}
subject to
\begin{eqnarray}
u(\ell^-) & = & u(\ell^+)\\
(1-\alpha) u^\prime(\ell^-) & = & \alpha u^\prime (\ell^+)
\end{eqnarray}
with initial datum $f$. The infinitesimal generator $A$ of $B^\alpha$ is given by
\begin{equation}
Au = \frac{\sigma^2}{2 a(x)}\nabla \left(a(x) \nabla u \right)
\end{equation}
where, for $\alpha \in (0,1)$,
\begin{equation}
a(x) = (1-\alpha) \mathbf{1}_{(-\infty,\ell]}(x) + \alpha \mathbf{1}_{(\ell,+\infty)}(x) = c_2 \exp \left( - c_1 \mathbf{1}_{(-\infty, \ell)}(x) \right)
\end{equation}
and $c_1= \ln (\alpha/(1-\alpha))$, $c_2=\alpha$. We evidently have that
\begin{equation*}
Pr\{ B^\alpha_t \leq \ell \} = 1-\alpha \qquad \textrm{and} \qquad  Pr\{ B^\alpha_t > \ell \}=\alpha.
\end{equation*}

\section{The modified skew Brownian motion on the line}
\subsection{The process $B^{(\alpha, \lambda)}_t$}

We focus on the process driven by the generator \eqref{gen-op-diff} with $b(\cdot)=0$ and coefficients
\begin{equation}
a(x) = (1-\alpha) \mathbf{1}_{(0,\ell]}(x) + \alpha \mathbf{1}_{(\ell,\infty)}(x), \quad \alpha \in (0,1)
\end{equation}
and 
\begin{equation}
\rho(x)a(x)=\sigma^2(x) = \mathbf{1}_{(0,\ell]}(x) + \lambda \mathbf{1}_{(\ell,\infty)}(x), \quad \lambda>0.
\end{equation}
Let us consider the interval $(l,r)$  where $r=\ell+\epsilon$ with $\epsilon \geq 0$ and $0< l < \ell < r$. We denote by $B^{(\alpha, \lambda)}_t$, $t>0$ the corresponding one-dimensional diffusion in $(l,r)$. The process $B^{(\alpha, \lambda)}$ moves like a skew Brownian motion  with reflecting barrier at $\ell$ and different variances inside $(l,r)$ (it respectively moves like the standard Brownian motion $B$ on $(l,\ell)$ and the Brownian motion $\sqrt{\lambda} B$ on $(\ell, r)$). The probability law $u$ of $B^{(\alpha, \lambda)}$ is the solution to
\begin{equation}
\partial_t u = A u= \left\lbrace \begin{array}{ll}
\displaystyle \frac{1}{2a(x)}\nabla (a(x) \nabla u) , & \textrm{in } (l,\ell)\\ 
\displaystyle \frac{\lambda}{2a(x)}\nabla (a(x) \nabla u), & \textrm{in } (\ell,r)
\end{array}  \right.
\label{opX}
\end{equation}
subject to
\begin{eqnarray}
 u(l) & = & 0 \label{bc1}\\
 u(r) & = & 0 \label{bc2}\\
  u(\ell^-) & = & u(\ell^+) \label{bc3}\\
(1- \alpha)u^\prime (\ell^-)  &=& \alpha u^\prime(\ell^+) \label{bc4}.
\end{eqnarray}

\begin{remark}
Notice that, if $\lambda = \alpha/(1-\alpha)$ then
\begin{equation*}
\rho(x) = \frac{\sigma^2(x)}{a(x)} = \frac{1}{1-\alpha}
\end{equation*}
and therefore
\begin{equation*}
Au= \frac{1}{2}\nabla( \sigma^2(x) \nabla u).
\end{equation*}
\end{remark}

\subsection{Scale function and Speed measure}

We define the scale function and the speed measure for the process $B^{(\alpha, \lambda)}_t$, $t>0$ previously introduced. For the sake of simplicity we consider $\ell=1$. From
\begin{equation}
\mathfrak{s}(x) = \frac{1}{a(x)}= \frac{1}{1-\alpha} \mathbf{1}_{(0,1]}(x) + \frac{1}{\alpha} \mathbf{1}_{(1,\infty)}(x)
\end{equation}
we get the scale function
\begin{equation}
S(x) = \frac{x}{1-\alpha} \mathbf{1}_{(0,1]}(x) + \left( \frac{1}{1-\alpha} + \frac{x-1}{\alpha} \right) \mathbf{1}_{(1,\infty)}(x)
\end{equation}
and recall that 
\begin{equation}
S(\eta_1, \eta_2]= S(\eta_2) - S(\eta_1).
\end{equation}
From the speed density
\begin{equation}
\mathfrak{m}(x) = \left( \sigma^2(x) \mathfrak{s}(x) \right)^{-1} = \left( \frac{1}{1-\alpha} \mathbf{1}_{(0,1]}(x) + \frac{\lambda}{\alpha} \mathbf{1}_{(1, \infty)}(x) \right)^{-1}
\end{equation}
we get the speed measure
\begin{equation}
M(x) = (1-\alpha) x \mathbf{1}_{(0,1]}(x) + \big( (1-\alpha) + \frac{\alpha}{\lambda} (x-1) \big) \mathbf{1}_{(1, \infty)}(x)
\end{equation}
with
\begin{equation}
M(\eta_1, \eta_2] = M(\eta_2) - M(\eta_1).
\end{equation}
Notice that for such a choice of $S(\cdot)$ and $M(\cdot)$ we obtain continuous functions. Moreover, the operator \eqref{gen-op-diff} can be written as
\begin{equation}
A= \frac{\sigma^2(x) \mathfrak{s}(x)}{2} \frac{d}{dx} \frac{1}{\mathfrak{s}(x)} \frac{d }{dx} = \frac{1}{2} \frac{d}{\mathfrak{m}(x)dx}\frac{d}{\mathfrak{s}(x)dx} = \frac{1}{2} \frac{d}{dM} \frac{d}{dS}  \label{L-true}
\end{equation}
and $(A, D(A))$ where
\begin{align*}
D(A) = \{g, Ag \in C_b\,:\, g \textrm{ satisfies } \eqref{bc1}, \eqref{bc2}, \eqref{bc3}, \eqref{bc4}\}
\end{align*}
is the infinitesimal generator of $B^{(\alpha, \lambda)}_t$, $t>0$, that is a skew diffusion on $(l,r)$ with no drift (that is $b=0$ in \eqref{gen-op-diff}), with transmission condition at $\ell=1$ and variances $1$ and $\lambda$ respectively on $(l,1)$ and $(1,r)$. We recall that $r=1+\epsilon$.

\subsection{Exit time from an interval}
Let $\tau_y = \inf\{ s \geq 0\,:\, B^{(\alpha, \lambda)}_s = y \}$ be the first time the process $B^{(\alpha, \lambda)}$ hits $y$. We also define 
$$ \tau_{(l,r)} := \tau_l \wedge \tau_r.$$ 
The solution to the Cauchy problem \eqref{opX} with initial datum $f$ can be written as follows
\begin{equation}
P^{(\alpha, \lambda)}_t f(x) = \mathbb{E}_x \left[ f(B^{(\alpha, \lambda)}_t) \mathbf{1}_{(t< \tau_{(l,r)})} \right].
\end{equation}
We use the fact that the probability function $\phi_\epsilon$ defined as
\begin{equation}
\phi_\epsilon(x) = \mathbb{P}\left( \tau_r < \tau_l \, | \, B^{(\alpha, \lambda)}_0=x \in (l,r) \right) = \mathbb{P}_x( \tau_r < \tau_l ) \label{prob-func-u}
\end{equation}
solves
\begin{eqnarray}
A \phi_\epsilon & = & 0\\
\phi_\epsilon(l) & = & 0\\
\phi_\epsilon(r) & = & 1
\end{eqnarray}
and, the problem to find $v_\epsilon \in D(A)$ such that $A v_\epsilon = - f$ can be solved by considering the function $\phi_\epsilon$, that is (\cite[page 197]{KarTay81})
\begin{align}
\label{KTsol}
v_\epsilon(x) = &  2\big(1-\phi_\epsilon(x)\big) \Sigma^{-}(x) + 2\phi_\epsilon(x) \Sigma^{+}(x) 
\end{align}
where
\begin{align}
\label{leSigma}
\Sigma^{-}(x) = \int_l^x S(l,\eta] f(\eta) \mathfrak{m}(\eta)d\eta \quad \textrm{and} \quad
\Sigma^{+}(x) = \int_x^r S(\eta, r] f(\eta) \mathfrak{m}(\eta) d\eta.
\end{align}
The probabilistic representation is given by
\begin{align*}
v_\epsilon(x) = \mathbb{E}_x \left[ \int_0^{\tau_{(l,r)}} f(B^{(\alpha, \lambda)}_t) dt \right].
\end{align*}
We recall that $\alpha=\alpha_\epsilon$ and $\lambda=\lambda_\epsilon$. For the sake of simplicity we consider $f \equiv 1$ and the mean exit time 
\begin{equation}
v_\epsilon(x) = \mathbb{E}\left[ \tau_{(l,r)} \, | \, B^{(\alpha, \lambda)}_0=x \in (l,r) \right] = \mathbb{E}_x\left[ \tau_{(l,r)} \right] \label{mean-exit-v}
\end{equation}
solves
\begin{eqnarray}
Av_\epsilon & = & -1, \quad v_\epsilon \in D(A).
\end{eqnarray}

We now obtain the explicit representation of $\phi_\epsilon$ and therefore of $v_\epsilon$. For the probability function \eqref{prob-func-u}, we get that (\cite{KarTay81}):
\begin{itemize}
\item[i)] for $(l,r) \subset (0,1)$ or $(l,r) \subset (1,\infty)$, 
\begin{equation}
\phi_\epsilon(x) = \frac{S(l,x]}{S(l,r]} = \frac{x-l}{r-l}, \quad x \in (l,r)
\end{equation}
that is, $B^{(\alpha, \lambda)}$ moves like a Brownian motion in each set;
\item[ii)] for $l <1$ and $r>1$ (that is case here, in our analysis),
\begin{align}
\phi_\epsilon(x) = & \frac{S(l,x]}{S(l,r]} = \frac{\alpha x \mathbf{1}_{[0,1)}(x) + \big( \alpha + (1-\alpha) (x-1)\big) \mathbf{1}_{(1, 1+\epsilon]}(x) - \alpha l}{(1-\alpha) (r-1) - \alpha (l-1)}, \quad x \in (l,r)\\
= &  \left\lbrace \begin{array}{ll}
\displaystyle  \frac{\alpha x - \alpha l}{(1-\alpha) (r-1) - \alpha (l-1)}, & x \in (l,1] \\
\displaystyle  \frac{(1-\alpha)(x-1) - \alpha (l-1) }{(1-\alpha) (r-1) - \alpha (l-1)}, & x \in (1,r)
\end{array} \right . .
\end{align}
\end{itemize}
For the mean exit time \eqref{mean-exit-v} we consider only the case ii), $0<l<1$ and $r>1$. 

\begin{proposition}
We have that
\begin{align}
v_\epsilon(x) = & \left\lbrace \begin{array}{ll}
\displaystyle 2\left(  \frac{I_1}{I_2} -  \frac{1}{\alpha}x \right) (\alpha x - \alpha l) + x^2 + l^2 - 2lx, & x \in (l,1]\\
\displaystyle 2\left( \frac{I_1}{I_2} - \frac{1}{\alpha} - \frac{(x-1)}{\lambda (1-\alpha)}  \right) \bigg( (1-\alpha)(x-1) -\alpha (l-1) \bigg) + \\
\displaystyle + 1 + l^2 + 2\frac{\alpha}{1-\alpha}\frac{(x-1)}{\lambda} - 2l - \frac{\alpha}{1-\alpha}\frac{2l}{\lambda} (x-1) + \frac{(x-1)^2}{\lambda}, & x \in (1,r)
\end{array} \right . \label{mean-exit-v-law}
\end{align}
where 
\begin{align*}
I_1= \frac{1-\alpha}{\alpha} (r-1) +\frac{(r-1)^2}{2\lambda} + \frac{(1-l^2)}{2}, \quad I_2=(1-\alpha)(r-1) - \alpha(l-1). 
\end{align*}
\end{proposition}

\begin{proof}
Consider \eqref{leSigma} with $f=1$, then
\begin{align*}
\Sigma^{+}=\Sigma(x,r] = & \int_x^r S(\eta, r] \mathfrak{m}(\eta)d\eta =  S(r)M(x,r]  - \int_x^r S(\eta) \mathfrak{m}(\eta)d\eta 
\end{align*}
and
\begin{align*}
\Sigma^{-}=\Sigma(l,x] = & \int_l^x S(l, \eta] \mathfrak{m}(\eta) d\eta = \int_l^x S(\eta) \mathfrak{m}(\eta) d\eta - S(l) M(l,x] 
\end{align*}
where the last integral can be explicitly written by considering that
\begin{align*}
\int_l^x S(\eta) \mathfrak{m}(\eta) d\eta = & \int_l^x S(\eta) \mathfrak{m}(\eta) d\eta \mathbf{1}_{(0,1]}(x) + \left( \int_l^1 S(\eta) \mathfrak{m}(\eta) d\eta + \int_1^x S(\eta) \mathfrak{m}(\eta) d\eta \right) \mathbf{1}_{(1, \infty)}(x)\\
= & \frac{(x^2 - l^2)}{2} \mathbf{1}_{(0,1]}(x) + \left( \frac{(1-l^2)}{2} + \frac{\alpha}{1-\alpha}\frac{(x-1)}{\lambda} + \frac{(x-1)^2}{2\lambda}  \right) \mathbf{1}_{(1, \infty)}(x)
\end{align*}
and 
$$S(l) M(l,x]  = -l^2 + lx\mathbf{1}_{(0,1]}(x) + \left( l+\frac{\alpha}{1-\alpha}\frac{l}{\lambda}(x-1) \right) \mathbf{1}_{(1, \infty)}(x).$$
Therefore, we have that
\begin{align*}
\Sigma(l,x] = & \left( \frac{x^2+l^2}{2} - lx \right) \mathbf{1}_{(0,1]}(x) \\
&  + \left( \frac{(1+l^2)}{2} + \frac{\alpha}{1-\alpha}\frac{(x-1)}{\lambda} - l  - \frac{\alpha}{1-\alpha}\frac{l}{\lambda}(x-1)  + \frac{(x-1)^2}{2\lambda} \right)  \mathbf{1}_{(1, \infty)}(x).
\end{align*}
Rewrite \eqref{KTsol} as follows
\begin{align}
\label{KTsolProof}
v_\epsilon(x) = & 2\phi_\epsilon(x) \bigg( \Sigma(x,r] - \Sigma(l,x] \bigg) + 2 \Sigma(l,x]
\end{align}
where
\begin{align*}
& \Sigma(x,r] - \Sigma(l,x] =  \\
= & S(r) M(x,r]  - \int_l^r S(\eta) \mathfrak{m}(\eta)d\eta + S(l) M(l,x] \\
= & S(r) \left( M(r) - M(x)\right)  - \frac{(1-l^2)}{2} - \frac{\alpha}{1-\alpha}\frac{(r-1)}{\lambda} -\frac{(r-1)^2}{2\lambda} + S(l) \left( M(x) - M(l) \right) \\
= & S(r) M(r) - S(l)M(l) - \left( S(r) - S(l) \right) M(x)  - \frac{(1-l^2)}{2} - \frac{\alpha}{1-\alpha}\frac{(r-1)}{\lambda} - \frac{(r-1)^2}{2\lambda}\\
= & I_1  -  \frac{I_2}{\alpha(1-\alpha)} M(x)
\end{align*}
with $I_2=(1-\alpha)(r-1) - \alpha (l-1)$, $S(l)M(l)=l^2$ and
\begin{align*}
I_1 = & S(r)M(r) - l^2- \frac{(1-l^2)}{2} - \frac{\alpha}{1-\alpha}\frac{(r-1)}{\lambda} - \frac{(r-1)^2}{2\lambda}\\
 = & \frac{1-\alpha}{\alpha} (r-1) +\frac{(r-1)^2}{2\lambda} + \frac{(1-l^2)}{2} 
\end{align*}
By collecting all pieces together, we obtain
\begin{align*}
v(x) = & \left\lbrace \begin{array}{ll}
\displaystyle 2\phi_\epsilon(x) \left(  I_1 -  \frac{I_2}{\alpha}x \right) + x^2 + l^2 - 2lx, & x \in (l,1]\\
\displaystyle 2\phi_\epsilon(x) \left( I_1 - \frac{I_2}{\alpha(1-\alpha)} \big( (1-\alpha) + \frac{\alpha}{\lambda}(x-1) \big) \right) + \\
\displaystyle + 1 + l^2 + 2\frac{\alpha}{1-\alpha}\frac{(x-1)}{\lambda} - 2l - \frac{\alpha}{1-\alpha}\frac{2l}{\lambda} (x-1) + \frac{(x-1)^2}{\lambda}, & x \in (1,r)
\end{array} \right .\\
= & \left\lbrace \begin{array}{ll}
\displaystyle 2\left(  \frac{I_1}{I_2} -  \frac{1}{\alpha}x \right) (\alpha x - \alpha l) + x^2 + l^2 - 2lx, & x \in (l,1]\\
\displaystyle 2\left( \frac{I_1}{I_2} - \frac{1}{\alpha} - \frac{1}{1-\alpha}\frac{(x-1)}{\lambda}  \right) \bigg( (1-\alpha)(x-1) -\alpha (l-1) \bigg) + \\
\displaystyle + 1 + l^2 + 2\frac{\alpha}{1-\alpha}\frac{(x-1)}{\lambda} - 2l - \frac{\alpha}{1-\alpha}\frac{2l}{\lambda} (x-1) + \frac{(x-1)^2}{\lambda}, & x \in (1,r)
\end{array} \right .
\end{align*}
which is the claim.
\end{proof}

\begin{remark}
For $l=0$ (and $r=1+\epsilon$), formula \eqref{mean-exit-v-law} becomes
\begin{align}
\label{v-eps-1sec-simple}
v_\epsilon(x) = & \left\lbrace \begin{array}{ll}
\displaystyle 2\alpha\frac{I_1}{I_2}x -  x^2, & x \in (0,1]\\
\displaystyle 2(1-\alpha)\frac{I_1}{I_2}(x-1) - 2\frac{1-\alpha}{\alpha}(x-1) -  \frac{(x-1)^2}{\lambda} + \\
\displaystyle  + 2\alpha \frac{I_1}{I_2} - 1 , & x \in (1,1+\epsilon)
\end{array} \right . 
\end{align}
where
\begin{align*}
I_1 =  \frac{1-\alpha}{\alpha} \epsilon +\frac{\epsilon^2}{2\lambda} + \frac{1}{2} , \quad I_2=(1-\alpha)\epsilon + \alpha.
\end{align*}
We get that $v_\epsilon \in D(A)$ satisfies $Av_\epsilon = -1$.
%\begin{eqnarray}
%Lv_\epsilon &=& -1\\
%v_\epsilon(0)&=&0\\
%v_\epsilon(1+\epsilon)&=&0\\
%v_{\epsilon}(1^-)&=&v_{\epsilon}(1^+)\\
%(1-\alpha)v^\prime_\epsilon(1^-) &= &\alpha v^\prime_\epsilon(1^+) 
%\end{eqnarray}
\end{remark}

\begin{remark}
\label{remarl-result-1dim}
Under \eqref{main-condition}, $v_\epsilon$  given in \eqref{v-eps-1sec-simple} converges pointwise in $\Omega=(0,1)$ to 
$$v(x) = C x  - x^2$$
where
\begin{align*}
C= \left\lbrace 
\begin{array}{ll}
\displaystyle 2, & \textrm{if } \alpha/\epsilon\to 0\\
\displaystyle \frac{2+G}{1+G}, & \textrm{if }  \alpha/\epsilon \to G \in (0,\infty)\\ 
\displaystyle 1, & \textrm{if }  \alpha/\epsilon \to \infty
\end{array}
\right . .
\end{align*}
We can immediately verify that $v$ satisfies Neumann, Robin and Dirichlet boundary condition at $\ell=1$ depending on the asymptotic behaviour of $\alpha/\epsilon$. Indee, we have that
\begin{align*}
\lim_{\epsilon \to 0} 2\alpha \frac{I_1}{I_2} = 2 \lim_{\epsilon \to \infty} \frac{1+ \frac{\epsilon}{2\lambda}\frac{\alpha}{1-\alpha} + \frac{1}{2}\frac{\alpha}{(1-\alpha)\epsilon}}{1 + \frac{\alpha}{(1-\alpha)\epsilon}}= \left\lbrace 
\begin{array}{ll}
\displaystyle 2, & \alpha/\epsilon\to 0\\
\displaystyle \frac{2+G}{1+G}, & \alpha/\epsilon \to G \in (0,\infty)\\ 
\displaystyle 1, & \alpha/\epsilon \to \infty
\end{array}
\right . .
\end{align*}
\end{remark}

\section{The modified skew Brownian motion on the plane}

\subsection{The process $R^{(\alpha, \lambda)}_t$}

Let $\Theta(t)$ be a Brownian motion on $\mathbb{S}^1_r=\{x \in \mathbb{R}^2\,:\, |x|=r \}$ where the radius $r=R(\cdot)$ is the $2$-dimensional Bessel process. Then, $\mathbf{B} \in \mathbb{R}^2$ can be written by considering the skew-product representation $(R, \Theta)$  where (see \cite[p. 269]{ItoMcK})
\begin{equation}
R(t), \; t>0 \quad \textrm{and} \quad \Theta\left( \int_0^t \left[ R(s) \right]^{-2}ds \right), \; t>0. \label{skew-rep}
\end{equation} 

Assume that $\mathbf{B}$ is starting at $x\in \mathcal{B}_\ell$, the disc of radius $\ell>0$. Let us consider the additive functional
\begin{equation}
\mathfrak{f}(t) = meas\{ 0\leq s\leq t\,:\, \mathbf{B}_s \in \mathcal{B}_\ell\}, \quad t>0
\end{equation}
which is the time the Brownian motion spends on the disc up to time $t$. The Bessel process running with the new clock $\mathfrak{f}^{-1}$, that is $R(\mathfrak{f}^{-1}(t))$, $t>0$, is the Bessel process reflecting at $\ell$. The Brownian motion in the disc $\mathcal{B}_\ell \subset \mathbb{R}^2$ reflecting at $\partial \mathcal{B}_\ell$ is identical in law to the skew representation above running with the clock $\mathfrak{f}^{-1}$. The spherical part $\Theta(\int_0^{\mathfrak{f}^{-1}(t)} (R(s))^{-2}d\mathfrak{f})$ is identical in law (\cite[p. 272]{ItoMcK}) to the  spherical part in  \eqref{skew-rep} with no time-change, that is the standard Brownian angle  sampled on the support of $\mathfrak{f}$ up to time $\mathfrak{f}^{-1}$. Therefore, due to the isotropic nature of the motion, the Brownian motion in the disc $\mathcal{B}_\ell$ of radius $\ell>0$ and reflecting at $\ell$ can be studied by considering only its time-changed radial part, that is $R(\mathfrak{f}^{-1}(t))$, $t>0$.

The Bessel process with parameter $\nu \in \mathbb{R}$ is a diffusion  with state space $(0,\infty)$ associated with the infinitesimal generator
\begin{equation}
\frac{d^2}{dx^2} + \frac{\nu -1}{x}\frac{d}{dx}.
\end{equation}
For $\nu=2$, the generator above can be written as
\begin{equation}
\label{gen-Bess-primadef}
\frac{1}{x}\frac{d}{dx}\left( x \frac{d}{dx}\right)
\end{equation}
and we denote by $R(t)=|\mathbf{B}(t)|$, $t>0$  the corresponding $2$-dimensional Bessel process which is the solution to  
\begin{equation}
dR(t) = \frac{1}{2}\frac{dt}{R(t)} + dB(t).
\end{equation}
If $R(t)$, $t>0$ starts from $0$, then it lives $0$ immediately, never to return. If the process starts from $x>0$, then it never hits $0$. If we consider the additive functional
\begin{equation}
\overline{\mathfrak{f}}(t) = meas\{ 0\leq s\leq t\,:\, \mathbf{B}_s \notin \mathcal{B}_\ell\}, \quad t>0
\end{equation}
we can construct a skew motion by exploiting the same reasoning as above. In particular, $\mathbf{B}$ time-changed by the inverse $\overline{\mathfrak{f}}^{-1}$ is a Brownian motion moving out the disc and reflecting on the boundary $\partial \mathcal{B}_\ell$ (from outside). By taking the $\alpha$ portion of trajectories out the disc and the $(1-\alpha)$ portion of trajectories in the disc, we construct a process which is identical in law to the skew Brownian motion on the plane with transmission condition on $\partial \mathcal{B}_\ell$. Since the skew-product allow us to consider only the time change of radial part $R$ (the spherical part is identical in law to the time-changed spherical part), we can study the skew Bessel process instead of the skew Brownian motion. Skew Bessel process has been considered for example in \cite{Blei12,DGS06}.

We formalize the problem and clarify the aspects mentioned above. Let $\Omega_\ell$ be the the disc of radius $\ell$ and let $\Omega_r$ be the disc or radius $r=\ell +\epsilon$, so that $\Omega_\ell \subset \Omega_r$. Moreover, we denote by $\Sigma_\epsilon = \Omega_r \setminus \overline{\Omega_\ell}$ the $\epsilon$-neighbourhood of $\partial \Omega_\ell$. Let $\Omega_l \subset \Omega_\ell$ be a disc of radius $0<l<\ell$. Our aim is to study the Brownian motion on $\Omega_r$ with Dirichlet conditions on $\partial \Omega_l$ and $\partial \Omega_r$ and transmission condition on $\partial \Omega_\ell$. Thus, we study a killed skew Brownian motion. Moreover, we require that the process exhibits different variances, we have a standard Brownian motion in $\Omega_\ell \setminus \overline{\Omega_l}$ and a Brownian motion with variance $\lambda t$ in $\Sigma_\epsilon$. The transmission condition is written as follows
\begin{eqnarray}
u|_{\partial \Omega^{-}_\ell} &=& u|_{\partial \Omega^{+}_\ell} \qquad \textrm{continuity on the boundary}, \label{BC1}\\
(1-\alpha) \partial_{\bf n} u |_{\partial \Omega^{-}_\ell} &=& \alpha \partial_{\bf n} u |_{\partial \Omega^{+}_\ell} \qquad \textrm{partial reflection}, \label{BC2}
\end{eqnarray}
where $\partial_{\bf n}u$ is the normal derivative of $u$. The infinitesimal generator $(A, D(A))$ is therefore given by
\begin{equation}
\label{A-form}
Au= \left\lbrace
\begin{array}{ll}
\displaystyle \frac{1}{2}\Delta u & \textrm{on } \Omega_\ell \setminus \overline{\Omega_l}\\
\displaystyle \frac{\lambda}{2}\Delta u & \textrm{on } \Omega_r \setminus \overline{\Omega_\ell}
\end{array}
\right.
\end{equation}
with
\begin{align}
\label{dom-A-form}
D(A) =\{g, Ag \in C_b\,:\, g|_{\partial \Omega_l} = g|_{\partial \Omega_r}=0,\, f\textrm{ satisfies } \eqref{BC1},\, \eqref{BC2}\}
\end{align} 
As in the previous section we assume that $\alpha=\alpha_\epsilon$, $\lambda=\lambda_\epsilon$ are parameters depending on $\epsilon$. We are interested in the asymptotic analysis (as $\epsilon\to 0$) for the solution 
\begin{align}
u \in D(A) \quad s.c.\ A u = -f
\end{align}
on the collapsing domain $\Omega_r$ under condition \eqref{main-condition}. Denote by $\tau_{l,r}$ the first exit time from $\Omega_r \setminus \overline{\Omega_l}$ of the planar Brownian motion $\mathbf{B}$ started at $x \in \Omega_r \setminus \overline{\Omega_l}$. Killing the process by $\tau_{l,r}$ we get a motion on $\Omega_r \setminus \overline{\Omega_l}$. Since we can find a time change such that $\mathbf{B}_{\mathfrak{f}^{-1}}$ is a skew Brownian motion with reflecting barrier on $\partial \Omega_\ell$, we consider the time-changed skew-product representation. The skew-product representation $(R_{\mathfrak{f}^{-1}}, \Theta_{\mathfrak{f}^{-1}})$ of $\mathbf{B}_{\mathfrak{f}^{-1}}$ can be replaced, for our purposes, by $(R_{\mathfrak{f}^{-1}}, \Theta_t)$ and we are allowed to study only the time-changed Bessel process. With \eqref{gen-Bess-primadef} in mind, the generator of the Bessel process we are interested in can be rewritten as follows
\begin{equation}
\label{gen-Bess-secondadef} 
Au = \frac{\sigma^2(x)}{2 a(x)} \nabla \left( a(x) \nabla u \right) = \left\lbrace
\begin{array}{ll}
\displaystyle \frac{1}{2} \left( \frac{d^2}{d x^2} + \frac{1}{x}\frac{d}{dx} \right)u & \textrm{on } (l, \ell)\\
\displaystyle \frac{\lambda}{2} \left( \frac{d^2}{d x^2} + \frac{1}{x}\frac{d}{dx} \right)u & \textrm{on } (\ell, r)
\end{array}
\right .
\end{equation} 
as will be discussed below.

\subsubsection{Scale function and Speed measure}

We consider the Bessel process started at $0$ with no drift which is characterized by the functions
\begin{equation}
\mathfrak{s}(x) = x^{-1} \quad \textrm{and} \quad \mathfrak{m}(x) = x, \qquad x \in (0,\infty).
\end{equation}
We notice that $\mathfrak{m}(\{0\})=0$ and $\{0\}$ is an entrance non-exit boundary. The scale function turns out to be
\begin{equation}
S(x) = \ln x.
\end{equation}
Our scope here is to consider the skew Bessel process started at zero and reflecting at $\ell>0$ with different variances before and after $\ell$. The coefficients are given by
\begin{equation}
\sigma^2(x) = \mathbf{1}_{(0,\ell]}(x) + \lambda \mathbf{1}_{(\ell, \infty)}(x) , \quad \lambda>0
\end{equation}
and 
\begin{equation}
a(x) = (1-\alpha) x \mathbf{1}_{(0,\ell]}(x) + \alpha x \mathbf{1}_{(\ell, \infty)}(x) .
\end{equation}
Thus, we obtain
\begin{equation}
\mathfrak{s}(x) = \frac{x^{-1}}{1-\alpha}\mathbf{1}_{(0,\ell]}(x) + \frac{x^{-1}}{\alpha}\mathbf{1}_{(\ell, \infty)}(x)
\end{equation}
and
\begin{equation}
\mathfrak{m}(x)= \left( \sigma^2(x)\mathfrak{s}(x) \right)^{-1} = (1-\alpha) x \mathbf{1}_{(0,\ell]}(x) + \frac{\alpha}{\lambda} x \mathbf{1}_{(\ell, \infty)}(x)
\end{equation}
with scale function
\begin{equation}
S(x) = \frac{\ln x}{1-\alpha} \mathbf{1}_{(0, \ell]}(x) + \left( \frac{\ln \ell}{1-\alpha} + \frac{\ln x/\ell}{\alpha}\right)\mathbf{1}_{(\ell, \infty)}(x)
\end{equation}
and speed measure
\begin{equation}
M(x) = \frac{1-\alpha}{2}x^2 \mathbf{1}_{(0,\ell]}(x) + \left( \frac{1-\alpha}{2}\ell^2 + \frac{\alpha}{2\lambda}(x^2 -\ell^2) \right) \mathbf{1}_{(\ell, \infty)}(x).
\end{equation}
The infinitesimal generator \eqref{gen-op-diff} of the modified skew Bessel process can be therefore rewritten as
\begin{equation}
L=\frac{1}{2}\frac{d}{dM} \frac{d}{dS}
\end{equation}
with $S$ and $M$ as above and (\cite[Theorem VII.3.12]{ RevYor99}) $Lg(\ell)=Lg(\ell^+)=Lg(\ell^-)$ with
\begin{equation}
D(L) = \left\lbrace g, Ag \in C_b([0,\infty))\,:\, \frac{dg^+}{dS}(0^+)=0,\, (1-\alpha)g^\prime(\ell^-) = \alpha g^\prime(\ell^+)  \right\rbrace
\end{equation}
where
\begin{equation*}
\frac{dg^+}{dS}(x) = \lim_{h \downarrow 0} \frac{g(x+h) - g(x)}{S(x+h)-S(x)}
\end{equation*}
is the $S$-derivative of $g$. Thus, we have the Bessel process on the positive real line with skew reflection at $\ell$. We denote by $R^{(\alpha, \lambda)}_t$, $t\geq 0$ the Bessel process on $(l,r)$ with generator $(A, D(A))$ where $A$ is given in \eqref{gen-Bess-secondadef} and 
\begin{equation}
D(A) = D(L) \cap \left\lbrace g \in C_b\, :\, g \textrm{ satisfies } g(l)=g(r)=0, \, g(\ell^-)=g(\ell^+) \right\rbrace.
\end{equation}

\subsection{Exit time from a disc}

We study the exit time of $B^{(\alpha, \lambda)}$ by considering the exit time of $R^{(\alpha, \lambda)}$.\\

Let us first consider the probability
\begin{equation}
\phi_\epsilon(x) = \mathbb{P}\left( \tau_r < \tau_l \, | \, R^{(\alpha, \lambda)}_0=x \in (l,r)  \right) = \mathbb{P}_x( \tau_r < \tau_l) \label{prob-func-u-R}
\end{equation}
where $\tau_y = \inf\{ s > 0\,:\, R^{(\alpha, \lambda)}_s = y \}$. For $0< l < \ell < r < \infty$, we obtain that
\begin{equation}
\phi_\epsilon(x) = \frac{S(x) - S(l)}{S(r) - S(l)} = \frac{1}{S(r) - S(l)} \left\lbrace \begin{array}{ll}
\displaystyle \frac{\ln x/l}{1-\alpha} , & x \in (l,\ell]\\
\displaystyle \frac{\ln \ell/l}{1-\alpha} + \frac{\ln x/\ell}{\alpha}, & x \in (\ell, r)
\end{array} \right .
\end{equation}
solves
\begin{eqnarray*}
A \phi_\epsilon &=& 0,\\
\phi_\epsilon(l)&=&0,\\
\phi_\epsilon(r)&=&1,\\
\phi_\epsilon(\ell^-)&=&\phi_\epsilon(\ell^+),\\
(1-\alpha)\phi_\epsilon^\prime (\ell^-)&=&\alpha \phi_\epsilon^\prime(\ell^+).
\end{eqnarray*}

Our aim is to find 
\begin{align}
\label{solBessTime}
v_\epsilon \in D(A) \; s.c.\ A v_\epsilon =-1.
\end{align}
The probabilistic representation is given in terms of the mean exit time
\begin{equation}
v_\epsilon(x) = \mathbb{E}\left[ \tau_{(l,r)} \, | \, R^{(\alpha, \lambda)}_0=x \in (l,r) \right] = \mathbb{E}_x \left[ \tau_{(l,r)} \right].  \label{mean-exit-v-R}
\end{equation}
Set
\begin{equation}
\Pi^*(r) = \left( \frac{1-\alpha}{2\alpha} - \frac{1}{2\lambda}\right) \ell^2 \ln \frac{r}{\ell} + \frac{\ell^2 -l^2}{4} + \frac{r^2 - \ell^2}{4\lambda}
\end{equation}
and
\begin{align*}
 \Pi_3(r)=\frac{\ln \ell/l}{1-\alpha} + \frac{\ln r/\ell}{\alpha}.
\end{align*}
We present the following explicit result.
\begin{proposition}
The solution to \eqref{solBessTime} is written as
\begin{align}
\label{v-epsilon-explicit}
v_\epsilon(x) = \left\lbrace \begin{array}{ll}
\displaystyle \frac{2\Pi^*(r)}{(1-\alpha) \Pi_3(r)} \ln x/l - \frac{x^2 -l^2}{2}, & x \in (l, \ell]\\
\displaystyle  \frac{2 \Pi^*(r)}{(1-\alpha)\Pi_3(r)} \ln \ell/l  - \frac{\ell^2 -l^2}{2} + \left( \frac{2\Pi^*(r)}{\alpha \Pi_3(r)} + \frac{\ell^2}{\lambda} - \frac{1-\alpha}{\alpha} \ell^2\right) \ln x/\ell - \frac{x^2 - \ell^2}{2\lambda}, & x \in (\ell, r)
\end{array} \right.
\end{align}
\end{proposition}
\begin{proof}
We use the same technique as in the previous section, based on the scale function and speed measure (see \cite{KarTay81}). From \eqref{leSigma} with $f=1$, we obtain
\begin{align*}
\Sigma^{+} = \Sigma(x,r] = & \int_x^r S(\eta, r] \mathfrak{m}(\eta)d\eta =  S(r)M(x,r]  - \int_x^r S(\eta) \mathfrak{m}(\eta)d\eta 
\end{align*}
and
\begin{align*}
\Sigma^{-} = \Sigma(l,x] = & \int_l^x S(l, \eta] \mathfrak{m}(\eta) d\eta = \int_l^x S(\eta) \mathfrak{m}(\eta) d\eta - S(l) M(l,x] 
\end{align*}
The solution $v_\epsilon$ can be obtained by calculation as in \eqref{KTsol} or \eqref{KTsolProof}. Set, for $x>\ell$, 
\begin{equation}
\Pi_1(x) = \frac{1}{2}\left( \ell^2 + \frac{\alpha}{1-\alpha} \frac{x^2 - \ell^2}{\lambda} \right)\ln \ell,
\end{equation}
\begin{equation}
\Pi_2(x)= \frac{1}{2}\left( \frac{1-\alpha}{\alpha} \ell^2 + \frac{x^2 -\ell^2}{\lambda} \right)\ln \frac{x}{\ell},
\end{equation}
\begin{equation}
\Pi_3(x) = \frac{\ln \ell/l}{1-\alpha} + \frac{\ln x/\ell}{\alpha},
\end{equation}
\begin{equation}
\Pi_4(x) = \frac{\ell^2}{2}\ln \ell - \frac{l^2}{2}\ln l - \frac{\ell^2-l^2}{4} + \frac{\alpha}{1-\alpha}\frac{x^2 -\ell^2}{2\lambda} \ln \ell + \frac{x^2}{2\lambda} \ln \frac{x}{\ell} - \frac{x^2-\ell^2}{4\lambda}.
\end{equation}
We can immediately verify that
\begin{equation}
S(r)M(r) =\Pi_1(r)+\Pi_2(r), \quad S(r) -S(l) =\Pi_3(r), \quad S(l)M(l) = \frac{l^2}{2}\ln l,
\end{equation}
\begin{equation}
S(l)M(x) = \frac{x^2}{2}\ln l \mathbf{1}_{(0,\ell]}(x) + \left( \frac{\ell^2}{2}\ln l + \frac{\alpha}{1-\alpha} \frac{x^2-\ell^2}{2\lambda} \ln l \right)\mathbf{1}_{(\ell, \infty)}(x),
\end{equation}
\begin{equation}
\int_l^r S(\eta)\mathfrak{m}(\eta)d\eta = \Pi_4(r) .
\end{equation}
Moreover, we have that
\begin{align*}
\Sigma(x,r] - \Sigma(l,x] = & S(r) M(x,r] - \int_l^r S(\eta)\mathfrak{m}(\eta)d\eta  + S(l)M(l,x]\\
= & S(r)M(r) - S(r)M(x) - \Pi_4(r) + S(l)M(x) - S(l)M(l)\\
= & \Pi_1(r) + \Pi_2(r) - \Pi_3(r) M(x) -\Pi_4(r) - S(l)M(l)
\end{align*}
and
\begin{align}
\Sigma(l,x] = & S(l)M(l) -S(l) M(x) + \left\lbrace \begin{array}{ll}
\displaystyle \frac{x^2}{2}\ln x - \frac{l^2}{2}\ln l - \frac{x^2 - l^2}{4}, & x \in (l,\ell]\\
\displaystyle \Pi_4(x), & x \in (\ell, r)
\end{array} \right .\\
= & \left\lbrace \begin{array}{ll}
\displaystyle \frac{x^2}{2} \ln x/l - \frac{x^2 - l^2}{4}, & x \in (l, \ell]\\
\displaystyle \frac{x^2}{2\lambda} \ln x/\ell - \frac{x^2 - \ell^2}{4\lambda} + \frac{\alpha}{1-\alpha}\frac{1}{2\lambda}(x^2 - \ell^2) \ln \ell/l
 + \frac{\ell^2}{2} \ln \ell/l - \frac{\ell^2 -l^2}{4}, & x \in (\ell, r)
 \end{array} \right.
\end{align}
Thus, by considering \eqref{KTsolProof} and $\phi_\epsilon$ in \eqref{prob-func-u-R},  we get that
\begin{equation}
v_\epsilon(x) = 2\phi_\epsilon(x)\big( \Sigma(x,r] - \Sigma(l,x] \big) + 2\Sigma(l,x] = v^*(x) + 2\Sigma(l,x]\label{sol-v-Bessel}
\end{equation}
where
\begin{align*}
v^*(x) = & \left\lbrace \begin{array}{ll}
\displaystyle \frac{2}{\Pi_3(r)}  \left( \Pi_1(r) + \Pi_2(r) - \Pi_3(r) M(x) -\Pi_4(r) - \frac{l^2}{2}\ln l \right)\frac{\ln x/l}{1-\alpha}, & x \in (l,\ell]		\\
\displaystyle \frac{2}{\Pi_3(r)}  \left( \Pi_1(r) + \Pi_2(r) - \Pi_3(r) M(x) -\Pi_4(r) - \frac{l^2}{2}\ln l \right) \Pi_3(x), & x \in (\ell, r)
\end{array} \right. .
\end{align*}
By setting
\begin{equation}
\Pi^*(r) = \Pi_1(r) + \Pi_2(r) - \Pi_4(r) - \frac{l^2}{2}\ln l = \left( \frac{1-\alpha}{2\alpha} - \frac{1}{2\lambda}\right) \ell^2 \ln \frac{r}{\ell} + \frac{\ell^2 -l^2}{4} + \frac{r^2 - \ell^2}{4\lambda},
\end{equation}
after straightforward manipulation we get the claim.
\end{proof}

\subsection{Asymptotics}

We consider now $v \in D(\mathcal{A})$ such that $\mathcal{A} v = -f$ given by
\begin{align}
\label{v-explicit}
v(x) = \mathbb{E}_x \left[ \int_0^\infty f(R^{(0)}_t) M^{(0)}_t dt \right]
\end{align}
where $(\mathcal{A},D(\mathcal{A}))$ is the generator of the Bessel diffusion on $(l, \ell)$. The multiplicative functional $M^{(0)}_t$ characterizes uniquely the process and agrees with the following choice for the generator $\mathcal{A}$:
\begin{align*}
\begin{array}{rl}
\displaystyle \textrm{i)} & \mathcal{A}=A_R, \;  D(A_R) = \{g, \mathcal{A}g \in C_b((l, \ell))\,:\, g(l)=0,\, g^\prime(\ell) + Gg(\ell)=0 \},\\
\displaystyle \textrm{ii)} & \mathcal{A}=A_N,\; D(A_N) = \{g, \mathcal{A}g \in C_b((l, \ell))\,:\, g(l)=0,\, g^\prime(\ell) =0 \},\\
\displaystyle \textrm{iii)} & \mathcal{A}=A_D,\; D(A_D) = \{g, \mathcal{A}g \in C_b((l, \ell))\,:\, g(l)=0,\, g(\ell) =0 \}.
\end{array}
\end{align*}
We provide the following result concerning \eqref{v-epsilon-explicit} and \eqref{v-explicit} for $f=1$.

\begin{theorem}
\label{thm-2dim-condition}
Under \eqref{main-condition}, $v_\epsilon \to v \in D(\mathcal{A})$ pointwise in $(l,\ell)$. In particular:
 \begin{itemize}
 \item[i)] $v_\epsilon \to v \in D(A_R)$ if $\alpha_\epsilon /\epsilon \to G \in (0,\infty)$;
 \item[ii)] $v_\epsilon \to v \in D(A_N)$ if $\alpha_\epsilon / \epsilon \to \infty$;
 \item[iii)] $v_\epsilon \to v \in D(A_D)$ if $\alpha_\epsilon / \epsilon \to 0$.
 \end{itemize}
\end{theorem}
\begin{proof}
Let us write
\begin{align*}
 \frac{2\Pi^*(r)}{(1-\alpha) \Pi_3(r)} =& \frac{\left( \frac{1-\alpha}{\alpha}\ln \frac{r}{\ell} - \frac{1}{\lambda} \ln \frac{r}{\ell}\right) \ell^2 + \frac{\ell^2 -l^2}{2} + \frac{2\ell \epsilon + \epsilon^2}{2\lambda} }{\ln \ell/l + \frac{1-\alpha}{\alpha}\ln r/\ell} 
\end{align*}
Throughout, we denote by $a_\epsilon \sim b_\epsilon$ the fact that $a_\epsilon/b_\epsilon \to 1$ as $\epsilon \to 0$. By considering that
\begin{align*}
\ln \left( 1+ \frac{\epsilon}{\ell} \right) \sim \frac{\epsilon}{\ell} - \frac{1}{2}\frac{\epsilon^2}{\ell^2}
\end{align*}
%
%\begin{align*}
%\frac{1}{\lambda} \ln r/\ell = \frac{\epsilon}{\lambda} \frac{1}{\epsilon} \ln \left(1 + \frac{\epsilon}{\ell}\right) =  \frac{\epsilon}{\lambda}  \ln \left(1 + \frac{\epsilon}{\ell}\right)^\frac{1}{\epsilon} \stackrel{\epsilon\to 0}{\longrightarrow} \frac{1}{\ell} \lim_{\epsilon\to 0} \frac{\epsilon}{\lambda}
%\end{align*}
%and
%\begin{align*}
%\frac{1-\alpha}{\alpha} \ln r/\ell = \frac{1-\alpha}{\alpha} \epsilon \frac{1}{\epsilon} \ln \left(1 + \frac{\epsilon}{\ell}\right) =  \frac{1-\alpha}{\alpha} \epsilon  \ln \left(1 + \frac{\epsilon}{\ell}\right)^\frac{1}{\epsilon} \stackrel{\epsilon\to 0}{\longrightarrow} \frac{1}{\ell} \lim_{\epsilon\to 0} \frac{1-\alpha}{\alpha} \epsilon
%\end{align*}
we get that
\begin{align*}
 \frac{2\Pi^*(\ell+\epsilon)}{(1-\alpha) \Pi_3(\ell+\epsilon)} \sim & \frac{\left( \frac{1-\alpha}{\alpha} \epsilon - \frac{\epsilon}{\lambda} \right) \ell +  \frac{\ell^2 -l^2}{2} + \frac{\ell \epsilon}{\lambda} + \frac{\epsilon^2}{2\lambda}}{\ln \ell/l + \frac{1-\alpha}{\alpha}\frac{\epsilon}{\ell}} = \frac{\frac{1-\alpha}{\alpha} \epsilon \ell +  \frac{\ell^2 -l^2}{2} + \frac{\epsilon^2}{2\lambda}}{\ln \ell/l + \frac{1-\alpha}{\alpha}\frac{\epsilon}{\ell}}
\end{align*}
which can be rewritten as
\begin{align}
\label{coefsimProof}
 \frac{2\Pi^*(\ell+\epsilon)}{(1-\alpha) \Pi_3(\ell+\epsilon)} \sim & \frac{\ell +  \frac{\ell^2 -l^2}{2} \frac{\alpha}{(1-\alpha)\epsilon}+ \frac{\alpha \epsilon}{2\lambda (1-\alpha)}}{\frac{\alpha}{(1-\alpha)\epsilon}\ln \ell/l +\frac{1}{\ell}}.
\end{align}
Therefore, under \eqref{main-condition}, from \eqref{coefsimProof} we have that:
\begin{itemize}
\item[i)] if $\alpha_\epsilon/\epsilon \to G> 0$, then
\begin{equation}
\label{C1}
 \frac{2\Pi^*(\ell+\epsilon)}{(1-\alpha) \Pi_3(\ell+\epsilon)} \to \frac{\frac{\ell}{G} + \frac{\ell^2 - l^2}{2} }{ \frac{1}{\ell G}+ \ln \ell/l },
\end{equation}
\item[ii)] if $\alpha_\epsilon/\epsilon \to 0$, then
\begin{equation}
\label{C2}
 \frac{2\Pi^*(\ell+\epsilon)}{(1-\alpha) \Pi_3(\ell+\epsilon)} \to \ell^2 ,
\end{equation}
\item[iii)] if $\alpha_\epsilon/\epsilon \to \infty$, then
\begin{equation}
\label{C3}
 \frac{2\Pi^*(\ell+\epsilon)}{(1-\alpha) \Pi_3(\ell+\epsilon)} \to \frac{\ell^2 -l^2}{2\ln \ell/l} .
\end{equation}
\end{itemize}
The solution \eqref{v-explicit} can be therefore written as 
$$v(x)= C \ln x/l - \frac{(x^2 -l^2)}{2}$$ 
where the coefficient $C$ is given by \eqref{C1} or \eqref{C2} or \eqref{C3} depending on the asymptotic behaviour of $\alpha_\epsilon/\epsilon$.
\end{proof}

\section{Proof of Theorem \ref{main-thm} }
\begin{proof}
We exploit the pointwise convergence of $v_\epsilon$ to $v$ and the characterization by multiplicative functionals of the corresponding semigroups. For the Brownian motion on the line we consider the result in Remark \ref{remarl-result-1dim} whereas for the Brownian motion on the plane we consider the result in Theorem \ref{thm-2dim-condition}. Then, the proof moves in both cases by following the same arguments. For this reason, we present the proof only for planar case.

From the previous Theorem \ref{thm-2dim-condition} we have the pointwise convergence
\begin{align*}
\mathbb{E}_x \left[ \int_0^\infty f(R^{(\alpha, \lambda)}_t) M^\epsilon_t dt \right] \to  \mathbb{E}_x \left[ \int_0^\infty f(R^{(0)}_t) M^{(0)}_t dt \right]
\end{align*}
for a constant $f$ where $M^{(0)}_t$ is determined by the boundary conditions on $(l,\ell)$ . Since the multiplicative functional characterizes uniquely the semigroup (\cite[Proposition 1.9]{BluGet68}) we can consider a non constant $f$. Under Robin boundary condition we have that
\begin{align*}
\mathbb{E}_x \left[ \int_0^\infty f(R^{(0)}_t) M^{(0)}_t dt \right] = & \mathbb{E}_x \left[ \int_0^\infty f(R^{(0)}_t) \mathbf{1}_{(t < \zeta \wedge \tau_l)} dt \right] \\
= & \mathbb{E}_x \left[ \int_0^\infty f(R^{(0)}_t) e^{- G L^{\{\ell\}}_t} \mathbf{1}_{(t < \tau_l)} dt \right] 
\end{align*}
where $\zeta$ is the (elastic) lifetime of $R^{(0)}$ such that $(t < \zeta) \equiv (L^{\{\ell\}}_t(R^{(0)}) < \gamma)$ and $\gamma$ is an independent exponential r.v. with parameter $G \in (0,\infty)$.  Observe that, we obtain $\mathcal{A}=A_R$ in Theorem \ref{thm-2dim-condition} under the hypothesis that \eqref{main-condition} holds and $\alpha_\epsilon/\epsilon \to G$, that is, trivially,
\begin{align*}
\mathbb{E}_x \left[ \int_0^\infty f(R^{(0)}_t) M^{(0)}_t dt \right] =  \mathbb{E}_x \left[ \int_0^\infty f(R^{(0)}_t) e^{- \left( \lim_{\epsilon\to 0} \frac{\alpha}{(1-\alpha)\epsilon} \right) L^{\{\ell\}}_t} \mathbf{1}_{(t < \tau_l)}  dt \right].
\end{align*}
The skew-product representation allows us to consider only the radial part of the process $B^{(\alpha, \lambda)}$. Thus, we can transfer the result about $R^{(\alpha, \lambda)}$ to $B^{(\alpha, \lambda)}$ and we prove that Theorem \ref{main-thm} holds under Robin boundary condition. 

By following the same argument as before, we can easily prove that Theorem \ref{main-thm} holds under Neumann boundary condition. Indeed, in this case we have that 
\begin{align*}
\forall\, t \geq 0, \quad \mathbb{P}_x( e^{-G L^{\{\ell\}}_t} =1 ) = 1
\end{align*}
and therefore, we get that
\begin{align*}
\mathbb{E}_x \left[ \int_0^\infty f(R^{(0)}_t) \mathbf{1}_{(t < \tau_l)} dt \right] =  \mathbb{E}_x \left[ \int_0^\infty f(R^{(0)}_t) e^{- \left( \lim_{\epsilon\to 0} \frac{\alpha}{(1-\alpha)\epsilon} \right) L^{\{\ell\}}_t} \mathbf{1}_{(t < \tau_l)}  dt \right].
\end{align*}
with $\alpha_\epsilon/\epsilon \to 0$.

Since $L^{\{\ell\}}_t$, $t\geq 0$ is positive continuous additive functional (PCAF), the process $\frac{\alpha}{(1-\alpha)\epsilon} L^{\{\ell\}}_t$, $t\geq 0$ is a PCAF. For $x \in (l, \ell)$, 
\begin{align*}
& \mathbb{E}_x\left[e^{- \frac{\alpha}{(1-\alpha)\epsilon} L^{\{\ell\}}_t} \mathbf{1}_{(t < \tau_l)}\right]\\
= & \mathbb{E}_x\left[e^{- \frac{\alpha}{(1-\alpha)\epsilon} L^{\{\ell\}}_t} \mathbf{1}_{(t < \tau_l)}, (t < \tau_\ell) \cup (t \geq \tau_\ell) \right]\\
= & \mathbb{E}_x\left[ e^{- \frac{\alpha}{(1-\alpha)\epsilon} L^{\{\ell\}}_t} \mathbf{1}_{(t < \tau_l)} \big| t < \tau_\ell \right] \mathbb{P}_x(t < \tau_\ell) + \mathbb{E}_x\left[ e^{- \frac{\alpha}{(1-\alpha)\epsilon} L^{\{\ell\}}_t}\mathbf{1}_{(t < \tau_l)} \big| t \geq \tau_\ell \right] \mathbb{P}_x(t \geq \tau_\ell)\\
= & \mathbb{E}_x\left[  \mathbf{1}_{(t < \tau_l)} \big| t < \tau_\ell \right] \mathbb{P}_x(t < \tau_\ell) + \mathbb{E}_x\left[ e^{- \frac{\alpha}{(1-\alpha)\epsilon} L^{\{\ell\}}_t}\mathbf{1}_{(t < \tau_l)} \big| L^{\{\ell\}}_t >0 \right] \mathbb{P}_x(t \geq \tau_\ell)\\
= & \mathbb{E}_x\left[  \mathbf{1}_{(t < \tau_l)} \mathbf{1}_{( t < \tau_\ell)} \right] + \mathbb{E}_x\left[ e^{- \frac{\alpha}{(1-\alpha)\epsilon} L^{\{\ell\}}_t},(t < \tau_l) \big| L^{\{\ell\}}_t >0 \right] \mathbb{P}_x(t \geq \tau_\ell)\\
= & \mathbb{E}_x\left[  \mathbf{1}_{(t < \tau_l \wedge \tau_\ell)} \right] + \mathbb{E}_x\left[ e^{- \frac{\alpha}{(1-\alpha)\epsilon} L^{\{\ell\}}_t},(t < \tau_l) \big| L^{\{\ell\}}_t >0 \right] \mathbb{P}_x(t \geq \tau_\ell).
\end{align*}
Then, we verify that
\begin{align*}
\mathbb{E}_x \left[ \int_0^\infty f(R^{(0)}_t) M^{(0)}_t dt \right] = & \mathbb{E}_x \left[ \int_0^\infty f(R^{(0)}_t) \mathbf{1}_{(t < \tau_l \wedge \tau_\ell)} dt \right] \\
= & \mathbb{E}_x \left[ \int_0^\infty f(R^{(0)}_t) e^{- \left( \lim_{\epsilon\to 0} \frac{\alpha}{(1-\alpha) \epsilon} \right) L^{\{\ell\}}_t} \mathbf{1}_{(t < \tau_l)} dt \right]
\end{align*}
if $\alpha_\epsilon/\epsilon \to \infty$.
\end{proof}

\section{The modified skew Brownian motion on $\mathbb{R}^d$}

Our results can be extended to diffusions in higher dimensions. We can follows different approaches depending on the regularity of the domains. For instance, in the case of $d$-dimensional balls we skip a detailed discussion about the Bessel process and underline only the fact that we can use the same argument as in the previous sections based on the speed measure and the scale function. Indeed, we can consider the skew-product representation and study the Bessel process with speed measure $\mathfrak{m}(x) = x^{d -1}$.  

Since we are interested in irregular domains we also approach the problem via Dirichlet form theory. More precisely, the problem $Au=-f$ where $A$ is given in \eqref{A-form} with domain \eqref{dom-A-form} can be formulated by considering
\begin{equation}
 -\frac{\sigma^2}{2 a(x)}\nabla \left(a(x) \nabla u \right)=f
\end{equation}
where, for $\alpha \in (0,1)$,
\begin{equation}
a(x) = (1-\alpha) \mathbf{1}_{\overline{\Omega_\ell} \setminus \overline{\Omega_l}}(x) + \alpha \mathbf{1}_{\Sigma_\epsilon}(x), \quad \alpha \in (0,1)
\end{equation}
and 
\begin{equation}
\rho(x)a(x)=\sigma^2(x) = \mathbf{1}_{\overline{\Omega_\ell} \setminus \overline{\Omega_l}}(x) + \lambda \mathbf{1}_{\Sigma_\epsilon}(x), \quad \lambda>0.
\end{equation}
We consider the measure 
$$dm=\frac1{\rho(x)}dx$$ 
(where we denoted by $dx$ the Lebesgue measure on $\mathbb{R}^d$) under the assumption that \eqref{main-condition} holds true. This ensure that
\begin{equation}
\label{mis} 
m(\Omega_{\ell +\epsilon}) \rightarrow  m(\Omega_\ell)\,\text{as}\,\epsilon\rightarrow 0.
\end{equation}
%Note that $m=\mathfrak{m}$ in $\mathbb{R}$. 

By multiplying by a test function $\psi \in H^1_0(\Omega_r \setminus \overline{\Omega_l})$ and integrating in $dm$ we have
$$-\int_{\Omega_r \setminus \overline{\Omega_l}} \frac{\sigma^2}{2 a(x)}\nabla \left(a(x) \nabla u \right) \psi dm=\int_{\Omega_r \setminus \overline{\Omega_l}} f \psi dm$$
$$-\int_{\Omega_r \setminus \overline{\Omega_l}} \nabla \left(a(x) \nabla u \right) \psi dx = 2\int_{\Omega_r \setminus \overline{\Omega_l}} f \psi dm$$
$$ (1-\alpha) \int_{\Omega_\ell \setminus \overline{\Omega_l}}  \nabla u  \nabla \psi dx +  \alpha  \int_{\Sigma_\epsilon} \nabla u  \nabla \psi dx = 2  (1-\alpha)  \int_{\Omega_\ell \setminus \overline{\Omega_l}} f \psi dx+ 2 \frac{\alpha }{\lambda} \int_{\Sigma_\epsilon}  f \psi dx.$$

Let $\Omega^*$ be an open regular domain  such that $\Omega^*\supset \overline{\Omega_{r}}.$ 
We recall that $\alpha=\alpha_\epsilon$ and $\lambda=\lambda_\epsilon$. Now we consider the sequence of energy functionals in $L^2(\Omega^*)$
\begin{equation}
F_\epsilon[u]=
\begin{cases}
\int_{\Omega_\ell \setminus \overline{\Omega_l}} (1-\alpha) |\nabla u|^2 dx+\alpha \int_{\Sigma _{\epsilon}}  |\nabla u|^2 dx &\text{if} \, u|_{\Omega_r \setminus \overline\Omega_l}\in H^1_0(\Omega_r \setminus \overline{\Omega_l})\\
+ \infty &\text{otherwise in }  L^2(\Omega^*).
\end{cases} \label{A(n)2}
\end{equation}
We study the Mosco convergence of \eqref{A(n)2}. First we recall the notion of $\emph{$M-$convergence}$ of functionals, introduced in \cite{MOS3}, (see also \cite{MOS1}).
\begin{definition}\label{def1}
A sequence of functionals $F_\epsilon :  H \rightarrow (-\infty,+\infty]$ is said to $M-$converge to a functional $F: H \rightarrow (-\infty,+\infty]$ in a Hilbert space $H$, if
\begin{itemize}
\item[(a)] For every $u \in H$ there exists $u_\epsilon$
 converging strongly to $u$ in  $H$ such that
\beq\label{aa} \limsup F_\epsilon[u_\epsilon] \leq F[u], \quad as\quad
\epsilon \rightarrow 0.
 \eeq
\item[(b)] For every $w_\epsilon$ converging weakly to $u$ in $H$
\beq\label{bb} \liminf F_\epsilon[w_\epsilon] \geq F[u], \quad as\quad
\epsilon \rightarrow 0. 
\eeq
\end{itemize}
\end{definition}

In the following we recall the results concerning the asymptotic behaviour of the functional \eqref{A(n)2} according to different rates of $\alpha/\epsilon$ (for the proof, see for example Theorem II.2 of \cite{AB} in the framework of $\Gamma$ convergence).
We recall the following results.  
\begin{theorem}
\label{teo1m} 
Under 
\begin{equation} 
\frac{\alpha}{\epsilon} \rightarrow G\,\text{as}\,\epsilon\rightarrow 0,
\end{equation}
the sequence of functionals $F_{\epsilon}$, defined in (\ref{A(n)2}), $M-$converges in $L^2(\Omega^*)$ to the functional 
\begin{equation}
F_{G}[u]=
\begin{cases}
\int_{\Omega_\ell \setminus \overline{\Omega_l}}  |\nabla u|^2 dx   + G \int_{\partial \Omega_\ell}
 u^2 ds &\text{if} \,u|_{\Omega_\ell \setminus \overline{\Omega_l}} \in H^1(\Omega_\ell \setminus \overline{\Omega_l}) \, \textrm{s.t. }  u=0 \, \textrm{ on } \partial \Omega_l  \\
+ \infty  &\text{otherwise in }  L^2(\Omega^*)
\end{cases} . \label{TeoA(u)22}
\end{equation}
\end{theorem}
 
\begin{theorem}
\label{teo2m} 
Under 
\begin{equation} 
\frac{\alpha}{\epsilon} \rightarrow 0 \,\text{as}\,\epsilon\rightarrow 0,
\end{equation}
the sequence of functionals $F_{\epsilon}$, defined in (\ref{A(n)2}), $M-$converges in $L^2(\Omega^*)$ to the functional 
\begin{equation}
F_{0}[u]=
\begin{cases}
\int_{\Omega_\ell \setminus \overline{\Omega_l}}  |\nabla u|^2 dx    & \text{if} \,u|_{\Omega_\ell \setminus \overline{\Omega_l}} \in H^1(\Omega_\ell \setminus \overline{\Omega_l}) \, \textrm{ s.t. } u=0 \, \textrm{ on } \partial \Omega_l  \\
+ \infty  &\text{otherwise in }  L^2(\Omega^*)
\end{cases} . \label{TeoA(u)22bis}
\end{equation}
\end{theorem}

\begin{theorem}
\label{teo3m} 
Under 
\begin{equation} 
\frac{\alpha}{\epsilon} \rightarrow \infty \,\text{as}\,\epsilon\rightarrow 0,
\end{equation}
the sequence of functionals $F_{\epsilon}$, defined in (\ref{A(n)2}), $M-$converges in $L^2(\Omega^*)$ to the functional 
\begin{eqnarray}
\label{Finf}
F_\infty[u]=\left\{
\begin{array}{lll}
\int_{\Omega_\ell \setminus \overline{\Omega_l}}  |\nabla u|^2 d x  &\text{if} \quad u|_{\Omega_\ell \setminus \overline{\Omega_l}} \in H_0^1(\Omega_\ell \setminus \overline{\Omega_l})\\
+ \infty \quad\quad\qquad\quad\quad\qquad\qquad &\text{otherwise in } L^2(\Omega^*). 
\end{array}
\right.  .
\end{eqnarray}
\end{theorem}
  
From the previous theorems we obtain the convergence of the solutions, see Theorem II.1, Theorems  II.2 and III.3 of \cite{AB}  (see also \cite{BCF}). We note that the following Lemma III.1  of  \cite{AB} plays a key role in this framework.
\begin{lemma}\label{l1}   Let $a:\R^d\to \R$ be a continuous function and $u\in H^{1}(\R^d)$
then $$\lim_{\epsilon\to 0}\frac1{\epsilon} \int_{\Sigma _{\epsilon}}  a(x) |u(x)|^2 dx= \int_{\partial \Omega_\ell}
 a(s) |u(s)|^2 ds.$$
\end{lemma}

We also remark that, by Lemma  \ref{l1},  we obtain that $$\lim_{\epsilon\to 0} \frac{\alpha }{\lambda} \int_{\Sigma_\epsilon}  f \psi dx = \lim_{\epsilon\to 0} \frac{\alpha \epsilon}{\lambda \epsilon} \int_{\Sigma_\epsilon}  f \psi dx= \lim_{\epsilon\to 0} \frac{\alpha \epsilon}{\lambda} \int_{\partial \Omega_\ell} f\psi ds$$
and by  \eqref{main-condition} we obtain $$\lim_{\epsilon\to 0} \frac{\alpha }{\lambda} \int_{\Sigma_\epsilon}  f \psi dx =0.$$

\begin{remark}
From the previous calculation  we obtain substantially different limit problems if $\alpha \epsilon/\lambda \to F \neq 0$. We will investigate this aspect in a forthcoming paper.
\end{remark}

\begin{remark}
Diffusions across fractal interfaces can be studied by using the approach based on Dirichlet form, in particular by extending Theorem \ref{teo1m}, Theorem \ref{teo2m} and Theorem \ref{teo3m}, see\cite{CV-asy}. This framework has been considered in \cite{pota} in order to study the asymptotic behaviour of corresponding multiplicative functionals only in the case that $\alpha=\lambda$. Our aim is therefore to extend such results on irregular domains in the case that $\alpha \neq \lambda$. 
\end{remark}

\begin{remark}
The approaches considered in the present paper can be also extended to Brownian motions time-changed by an inverse to a stable subordinator. That is, we can consider a delayed Brownian motion and it is our feeling that the asymptotic analysis could be deeply affected by the delay.
\end{remark}


\begin{thebibliography}{58}
\bibitem{ABTWW11}
{\sc T. Appuhamillage, V. Bokil, E. Thomann, E. Waymire,  B. Wood, }
\emph{Occupation and local times for skew Brownian motion with applications to dispersion across an interface.}
Ann. Appl. Probab., 21 (2011), 1, 183 - 214.

\bibitem{AB}
{\sc E.  Acerbi, G. Buttazzo,}
\emph{Reinforcement problems in the calculus of variations},  Ann. Inst. H. Poincar\'e Anal. Non Lin. 3  (1986),  no. 4, 273--284. 

\bibitem{BCF}
{\sc H. Brezis, L.A. Caffarelli, A. Friedman,} 
\emph{Reinforcement problems for elliptic equations and variational inequalities},  
Ann. Mat. Pura Appl. (4)  123  (1980), 219--246.
  
\bibitem{Blei12}
{\sc S. Blei,}
\emph{On symmetric and skew Bessel processes.}
Stochastic Processes and their Applications, 122 (2012), 3262 - 3287.


\bibitem{BluGet68} {\sc R.M. Blumenthal, R.K. Getoor,}
\emph{Markov Processes and Potential Theory}, 
Academic Press, New York, 1968.

\bibitem{pota}
{\sc R. Capitanelli, M. D'Ovidio,}
\emph{Skew Brownian diffusions across Koch interfaces},
Potential Anal. (2016) doi: 10.1007/s11118-016-9588-4.

\bibitem{CV-asy} 
{ \sc R. Capitanelli, M.A. Vivaldi,}
\emph{On the Laplacean transfer across fractal mixtures},
Asymptot. Anal. 83 (2013), no. 1-2, 1-33.

\bibitem{DGS06}
{\sc M. Decamps , M. Goovaerts, W. Schoutens,}
\emph{Asymmetric skew Bessel processes and their applications to finance.}
Journal of Computational and Applied Mathematics, 186 (2006), 130-147.

\bibitem{dovSL}
{\sc M. D'Ovidio,}
\emph{From Sturm-Liouville problems to fractional and anomalous diffusions.}
Stochastic Process. Appl. 122 (2012), no. 10, 3513-3544. 

\bibitem{HarrShep81}
{\sc J. M. Harrison, L. A. Shepp,}
\emph{On Skew Brownian Motion.}
Ann. Probab., 9 (1981), 309 - 313.

\bibitem{KarTay81}
{\sc S.Karlin, H. M. Taylor,}
\emph{A Second Course in Stochastic Processes},
Academic Press, New York, 1981.

\bibitem{ItoMcK}
{\sc K. It\^{o}, H. P. McKean, Jr., }
\emph{Diffusion Processes and Their Sample Paths}. 
Springer-Verlag, Heidelberg New York 1974.

\bibitem{Lejay06}
{\sc A. Lejay, }
\emph{On the constructions of the skew Brownian motion. }
Probab. Surveys, 3 (2006), 413-466.

\bibitem{LMS}
{\sc N. N. Leonenko, M. M. Meerschaert, A. Sikorskii,}
\emph{Fractional Pearson diffusions.} 
J. Math. Anal. Appl., 403 (2013), no. 2, 532-546

\bibitem{LS}
{\sc N. N. Leonenko, N. \v{S}uvak,}
\emph{Statistical inference for reciprocal gamma diffusion process.}
J. Statist. Plann. Inference 140 (2010), no. 1, 30-51


\bibitem{MOS3} 
{\sc U. Mosco,}
\emph{Convergence of convex sets and of solutions of
variational inequalities}, Adv. in Math. { 3}, (1969), 510-585.

\bibitem{MOS1} 
{\sc U. Mosco,}
\emph{{Composite media and asymptotic Dirichlet forms},}
J. Funct. Anal., { 123}, (1994), no.2, 368-421.

\bibitem{ORT13}
{\sc Y. Ouknine, F. Russo, G. Trutnau,}
\emph{On countably skewed Brownian motion with accumulation point.}
arXiv:1308.0441
%
%\bibitem{Pes06}
%G. Pe\v{s}kir.
%On the fundamental solution of the Kolmogorov-Shiryaev equation.
%\emph{From stochastic calculus to mathematical finance}, pages 535 - 546, 2006. Springer, Berlin.
%
\bibitem{Ram11}
{\sc J. M. Ramirez,}
\emph{Multi-skewed Brownian motion and diffusion in layered media.}
Proceedings of the American Mathematical Society, 139 (2011), 3739-3752.

\bibitem{RevYor99}
{\sc D. Revuz, M. Yor,}
\emph{Continuous Martingales and Brownian Motion}, 
3rd edn, Berlin: Springer-Verlag, 1999.

\bibitem{walsh78}
{\sc J.B. Walsh,}
\emph{A diffusion with a discontinuous local time.}
Asterisque, (1978) 52 - 53, 37-45.
\end{thebibliography}
\end{document}